\documentclass[4pt]{article} 

\usepackage[utf8]{inputenc} 
\usepackage{geometry} 
\usepackage{graphicx} 

\usepackage[colorlinks,citecolor=red]{hyperref}
\usepackage{amsmath}
\usepackage{amsfonts}
\usepackage{dsfont}
\usepackage{mathrsfs}
\usepackage{latexsym}
\usepackage{graphicx}
\usepackage{amssymb}
\usepackage{float}
\usepackage{epsfig}
\usepackage{epstopdf}
\usepackage{color,xcolor}
\usepackage{booktabs} 
\usepackage{array} 
\usepackage{paralist} 
\usepackage{verbatim} 
\usepackage{subfig} 

\newtheorem{theorem}{Theorem}[section]
\newtheorem{lemma}[theorem]{Lemma}

\usepackage{fancyhdr} 
\usepackage{sectsty}
\allsectionsfont{\sffamily\mdseries\upshape} 

\usepackage[nottoc,notlof,notlot]{tocbibind} 
\usepackage[titles,subfigure]{tocloft} 

\title{Wiener index and Harary index on pancyclic graphs}
\author{Huicai Jia \thanks{College of Science, Henan
University of Engineering, Zhengzhou, Henan 451191, China;
School of Mathematics, Renmin University of China, Beijing, 100872, China. Email: jhc607@163.com},
Hongye Song \thanks{School of General Education, Beijing International Studies University, Beijing, China. Email: songhongye@bisu.edu.cn}}
\date{} 

\begin{document}
\maketitle

\begin{abstract}
Wiener index and Harary index are two classic and well-known topological indices for the characterization of molecular graphs. Recently, Yu et al. \cite{YYSX} established some sufficient conditions for a graph to be pancyclic in terms of the edge number, the spectral radius and the signless Laplacian spectral radius of the graph. In this paper, we give some sufficient conditions for a graph being pancyclic in terms of the Wiener index, the Harary index, the distance spectral radius and the Harary spectral radius of a graph.


\end{abstract}

\section{Introduction}
\noindent All graphs considered here are finite, undirected and connected simple graphs. Let $G$ be a graph with vertex set
$V(G)=\{v_{1},v_{2},\ldots,v_{n}\}$ and edge set $E(G)$. Let $N_{G}(v)$ denote the neighbour set of $v$ in $G$.
Denote by $d_{i}$ or $d(v_{i})$ the degree of vertex $v_{i}$ . Let $(d_{1},d_{2},\ldots,d_{n})$ be the degree
sequence of $G,$ where $d_{1}\leq d_{2}\leq \cdots \leq d_{n}.$ Then $d_{1}=\delta$ is called the minimum degree of $G.$
We use $d_{G}(v_{i},v_{j})$ to denote the distance between vertices $v_{i}$ and $v_{j}$. The {\it union} of simple graphs $G$ and $H$ is the graph
$G\cup H$ with vertex set $V(G)\cup V(H)$ and edge set $E(G)$ $\cup$ $E(H)$. If $G$ and $H$
are disjoint, we refer to their union as a disjoint union, and denote it by $G+H.$ The
disjoint union of $k$ graphs $G$ is denoted by $kG$. By starting with a disjoint union of $G$ and $H$ and adding edges joining every vertex of $G$ to every vertex of $H$, one can obtain the {\it join}
of $G$ and $H$, denoted by $G\vee H$. Let $\overline{G}$ denote the complement of $G$.

A pancyclic graph is a graph that contains cycles of all possible lengths from three up to the number of vertices in the graph. Pancyclic graphs are a generalization of Hamiltonian graphs, which have a cycle of the maximum possible length.

The distance matrix $D(G)$ is defined so that $(i,j)$-entry, $d_{ij}$,
is equal to $d_{G}(v_{i},v_{j})$. The {\it Wiener index} of a molecular graph was introduced by and named by Wiener \cite{WH}
in 1947. It is defined as the sum of distances between all pairs of vertices of a connected graph $G$, i.e.,
$W(G)=\sum_{{u,v}\in V(G)}d_{G}(u,v).$
Let $D_{i}(G)$ and $D_{v}(G)$ denote the sum of row $i$ of $D(G)$ and
the row sum of $D(G)$ corresponding to vertex $v,$ respectively. Then
$$W(G)=\frac{1}{2}\sum_{v\in V(G)}D_{v}(G)=\frac{1}{2}\sum_{i=1}^{n}D_{i}(G).$$ The {\it distance spectral radius} of $G$ is the largest
eigenvalue of $D(G)$, denoted by $\rho(G)$.

The Harary index of a graph $G,$ denoted by $H(G),$ has been introduced independently
by Ivanciuc et al. \cite{IO} and Plav\v{s}i\'{c} et al. \cite {PD} in 1993 for the characterization of
molecular graphs. The {\it Harary index} $H(G)$ is defined as the sum of reciprocals of distances between
all pairs of vertices of a connected graph $G$, i.e.
$H(G)=\sum_{{u,v}\in V(G)}\frac{1}{d_{G}(u,v)}.$
Note that in any disconnected graph $G,$ the distance is infinite between any two vertices from two distinct components.
Therefore its reciprocal can be viewed as zero. Thus we can define validly the Harary index of disconnected
graph $G$ as follows:
$H(G)=\sum_{i=1}^{k}H(G_{i}),$
where $G_{1}, G_{2},\ldots, G_{k}$ are the components of $G$. We often use $\hat{D}_{i}(G)$ or $\hat{D}_{v_{i}}(G)$ to denote $\sum_{v_{j}\in V(G)}\frac{1}{d_{G}(v_{i},v_{j})}$. Then
$$H(G)=\frac{1}{2}\sum_{v_{i}\in V(G)}\hat{D}_{v_{i}}(G)=\frac{1}{2}\sum_{i=1}^{n}\hat{D}_{i}(G).$$
The Harary matrix $RD(G)$ of $G$, which is initially called the reciprocal distance matrix and introduced by \cite{IO}, is
an $n\times n$ matrix whose $(i,j)$-entry is equal to $\frac{1}{d_{ij}}$
if $i\neq j$ and 0 otherwise. The {\it Harary spectral radius} of $G$ is the largest
eigenvalue of $RD(G)$, denoted by ${\rho}^{\star}(G)$. Note that in any disconnected graph $G,$ ${\rho}^{\star}(G)=\max\{{\rho}^{\star}(G_{i})| 1\leq i\leq k\}.$

The problem of deciding whether a given graph is Hamiltonian is
NP-complete. Many reasonable sufficient or necessary conditions were given. Recently, spectral graph theory has been applied to the problem. Some related sufficient spectral conditions on the Wiener index and the Harary index
for a graph to be Hamiltonian, traceable and Hamiltonian-connected have been given in \cite{FLH, HH, HN, DX3, DH, LR1, LR2, DX1, DX2, YL, ZT}.

In this paper, we consider the problem of deciding whether a given graph is pancyclic. Yu et al. \cite{YYSX} established some sufficient conditions for a graph to be pancyclic in terms of the edge number, the spectral radius and the signless Laplacian spectral radius of the graph. Motivated by these results, we present some sufficient conditions for a graph to be pancyclic in terms of the Wiener index, the Harary index, the distance spectral radius and the Harary spectral radius of a graph, respectively.

\section{\mathversion{bold} Preliminaries}

\noindent Before giving the proof of our theorems, we introduce some
fundamental lemmas and properties in this section.

Let $\mathbb{N}\mathbb{P}=\{K_{2}\vee (K_{n-4}+2K_{1}), K_{5}\vee6K_{1}, K_{3}\vee
(K_{2}+3K_{1}), K_{3}\vee (K_{1, 4}+K_{1}), K_{3}\vee (K_{1, 3}+K_{2}), (K_{2}\vee 2K_{1})\vee 5K_{1}, K_{4}\vee5K_{1}, K_{1, 2}\vee
4K_{1}, K_{2}\vee (K_{1, 3}+K_{1}), K_{3}\vee4K_{1} \}$.

\begin{lemma}{\bf (\cite{YYSX})}\label{le1}
Let $G$ be a connected graph on $n\geq5$ vertices and $m$ edges with minimum degree $\delta\geq2$. If
$m\geq\binom{n-2}{2}+4$, then $G$ is a pancyclic graph unless $G\in\mathbb{N}\mathbb{P}$ or $G$ is a bipartite graph.
\end{lemma}

\begin{lemma}{\bf (\cite{IG})} \label{le2}
Let $G$ be a connected graph on $n$ vertices. Then $\rho(G)\geq\frac{2W(G)}{n},$ and the equality holds
if and only if the row sums of $D(G)$ are all equal.
\end{lemma}

\begin{lemma}{\bf(\cite{DX1})} \label{le3}
Let $G$ be a graph on $n$ vertices. Then ${\rho}^{\star}(G)\geq\frac{2H(G)}{n},$ and the equality holds
if and only if the row sums of $RD(G)$ are all equal.
\end{lemma}

\begin{lemma}\label{le4}
Let $G$ be a bipartite graph on $n$ vertices. Then $W(G)>\frac{n^{2}+3n-14}{2}$.
\end{lemma}

\begin{proof}
Let $G$ be a bipartite graph of order $n$ whose vertices are divided into two disjoint and independent sets $V_{1}$ and $V_{2}$, $|V_{1}|=a, (1\leq a\leq n-1),$ $|V_{2}|=n-a.$\\
\begin{eqnarray}
\nonumber
W(G)& = & \frac{1}{2}\sum_{i\in V(G)}D_{i}(G)
\\ \nonumber
& \geq & \frac{1}{2}[\sum_{i\in V_{1}}(d_{i}+3(n-a-d_{i})+2(a-1))+\sum_{j\in V_{2}}(d_{j}+3(a-d_{j})+2(n-a-1))]
\\ \nonumber
& =& \frac{1}{2}[\sum_{i\in V_{1}}(3n-a-2-2d_{i})+\sum_{j\in V_{2}}(2n+a-2-2d_{j})]
\\ \nonumber
& = & -a^{2}+na+n^{2}-n-\sum_{i\in V(G)}d_{i}
\\ \nonumber
& = & -a^{2}+na+n^{2}-n-2
\\ \nonumber
& \geq & -a^{2}+na+n^{2}-n-2a(n-a)=a^{2}-na+n^{2}-n
\\
& \geq & \frac{3n^{2}-4n}{4}.
\end{eqnarray}

By manipulation, we have $\frac{3n^{2}-4n}{4}>\frac{n^{2}+3n-14}{2}$. Then $W(G)>\frac{n^{2}+3n-14}{2}$. The proof of Lemma \ref{le4} is completed.
\end{proof}

\begin{lemma}\label{le5}
Let $G$ be a bipartite graph on $n$ vertices. Then $H(\overline{G})>\frac{5n^{2}-23n+28}{2(n-1)}$ for $n\geq8$.
\end{lemma}

\begin{proof}
Let $G$ be a bipartite graph of order $n\geq8$ whose vertices are divided into two disjoint and independent sets $V_{1}$ and $V_{2}$, $|V_{1}|=a, (1\leq a\leq n-1),$ $|V_{2}|=n-a.$

\noindent {\bf Case 1.} $G$ is not a complete bipartite graph.
\begin{eqnarray*}
H(\overline{G}) & = & \frac{1}{2}\sum_{i=1}^{n}\hat{D}_{i}(\overline{G})=\frac{1}{2}\sum_{v\in V(G)}[d_{\overline{G}}(v)+\frac{1}{2}(n-1-d_{\overline{G}}(v))]\\
& = & \frac{1}{2}\sum_{v\in V(G)}[\frac{n-1}{2}+\frac{d_{\overline{G}}(v)}{2}]=\frac{n(n-1)}{4}+\frac{1}{4}\sum_{v\in V(G)}d_{\overline{G}}(v)\\
& = & \frac{n(n-1)}{4}+\frac{1}{4}\sum_{v\in V(G)}(n-1-d_{G}(v))\\
& = & \frac{n(n-1)}{2}-\frac{m}{2} \\
& > & \frac{n(n-1)}{2}-\frac{1}{2}a(n-a)\\
& = & \frac{a^{2}-na+n^{2}-n}{2}\\
& \geq & \frac{3n^{2}-4n}{8}
\end{eqnarray*}

Through calculation, we have $\frac{3n^{2}-4n}{8}>\frac{5n^{2}-23n+28}{2(n-1)}.$ Then $H(\overline{G})>\frac{5n^{2}-23n+28}{2(n-1)}$ for $n\geq3.$

\noindent {\bf Case 2.} $G$ is a complete bipartite graph.

Then $G=K_{a, n-a}$ and
\begin{eqnarray*}
H(\overline{G})&=&\frac{1}{2}\sum_{i=1}^{n}\hat{D}_{i}(\overline{G})=\frac{1}{2}[\sum_{v\in V_{1}}(a-1)+\sum_{v\in V_{2}}(n-a-1)]\\
&=&\frac{1}{2}[a(a-1)+(n-a)(n-a-1)]\\
&=&a^{2}-na+\frac{n^{2}-n}{2}\\
&\geq&\frac{n^{2}-2n}{4}
\end{eqnarray*}

By manipulation, we have $\frac{n^{2}-2n}{4}>\frac{5n^{2}-23n+28}{2(n-1)}$ for $n\geq8.$ Then $H(\overline{G})>\frac{5n^{2}-23n+28}{2(n-1)}$ for $n\geq8.$ The proof of Lemma \ref{le5} is completed.
\end{proof}

\section{Wiener index, Harary index on pancyclic graphs}

\begin{theorem}\label{th6}
Let $G$ be a connected graph of order $n\geq5$ with minimum degree $\delta\geq2$. If $W(G)\leq\frac{n^{2}+3n-14}{2},$ then $G$ is a pancyclic graph unless
$G\in\mathbb{N}\mathbb{P}$.
\end{theorem}

\begin{proof}
Suppose that $G$ is not a pancyclic graph. Then
\begin{eqnarray}
\nonumber
W(G)&=&\frac{1}{2}\sum_{i=1}^nD_{i}(G)\geq\frac{1}{2}\sum_{i=1}^n(d_{i}+2(n-1-d_{i}))
\\ \nonumber
&=&\frac{1}{2}\sum_{i=1}^n(2(n-1)-d_{i})\\
&=&n(n-1)-m.
\end{eqnarray}
Note that $W(G)\leq\frac{n^{2}+3n-14}{2}$, we have $m\geq n(n-1)-\frac{n^{2}+3n-14}{2}=\binom{n-2}{2}+4.$ By Lemma \ref{le1}, we obtain that $G\in\mathbb{N}\mathbb{P}$ or $G$ is a bipartite graph. By a direct calculation, for all $G\in\mathbb{N}\mathbb{P}$, $W(G)\leq\frac{n^{2}+3n-14}{2}.$
If $G$ is a bipartite graph, by Lemma \ref{le4}, we obtain $W(G)>\frac{n^{2}+3n-14}{2}.$ This completes the proof.
\end{proof}

\begin{theorem}\label{th7}
Let $G$ be a connected graph of order $n\geq5$ with minimum degree $\delta\geq2$, and $\overline{G}$ be a connected graph. If $W(\overline{G})\geq\frac{n^{3}-6n^{2}+23n-28}{2},$ then $G$ is a pancyclic graph unless
$G$ is a bipartite graph which does not contain complete bipartite graph.
\end{theorem}

\begin{proof}
Suppose that $G$ is not a pancyclic graph. Then
\begin{eqnarray*}
W(\overline{G})&=&\frac{1}{2}\sum_{i=1}^nD_{i}(\overline{G})\leq\frac{1}{2}\sum_{v\in V(G)}[d_{\overline{G}}(v)+(n-1)(n-1-d_{\overline{G}}(v))]\\
&=&\frac{1}{2}\sum_{v\in V(G)}[(n-1)^{2}+(2-n)d_{\overline{G}}(v)]\\
&=&\frac{1}{2}n(n-1)^{2}-\frac{n-2}{2}\sum_{v\in V(G)}(n-1-d_{G}(v))\\
&=&\frac{n(n-1)}{2}+(n-2)m.
\end{eqnarray*}
Note that $W(\overline{G})\geq\frac{n^{3}-6n^{2}+23n-28}{2}$, we have $m\geq\frac{n^{3}-7n^{2}+24n-28}{2(n-2)}=\binom{n-2}{2}+4.$
By Lemma \ref{le1}, we obtain that $G\in\mathbb{N}\mathbb{P}$ or $G$ is a bipartite graph. Note that for all $G\in\mathbb{N}\mathbb{P}$, $\overline{G}$ is disconnected, and $G$ is a complete bipartite graph, $\overline{G}$ is also disconnected. The proof of Theorem \ref{th7} is completed.
\end{proof}

\begin{theorem}\label{th8}
Let $G$ be a connected graph of order $n\geq5$ with minimum degree $\delta\geq2$. If $H(G)\geq\frac{n^{2}-3n+7}{2},$ then $G$ is a pancyclic graph unless
$G\in\mathbb{N}\mathbb{P}$.
\end{theorem}

\begin{proof}
Suppose that $G$ is not a pancyclic graph. Then
\begin{eqnarray*}
H(G)&=&\frac{1}{2}\sum_{i=1}^n\hat{D}_{i}(G)\leq\frac{1}{2}\sum_{i=1}^{n}(d_{i}+\frac{1}{2}(n-1-d_{i}))\\
&=&\frac{n(n-1)}{4}+\frac{m}{2}.
\end{eqnarray*}
Note that $H(G)\geq\frac{n^{2}-3n+7}{2}$, we have $m\geq n^{2}-3n+7-\frac{n(n-1)}{2}=\binom{n-2}{2}+4.$
By Lemma \ref{le1}, we obtain that $G\in\mathbb{N}\mathbb{P}$ or $G$ is a bipartite graph. By a direct calculation, for all $G\in\mathbb{N}\mathbb{P}$, $H(G)\geq\frac{n^{2}-3n+7}{2}.$
If $G$ is a bipartite graph whose vertices are divided into two disjoint and independent sets $V_{1}$ and $V_{2}$, $|V_{1}|=a, (1\leq a\leq n-1),$ $|V_{2}|=n-a.$ We have
\begin{eqnarray*}
H(G)&=&\frac{1}{2}\sum_{i\in V(G)}\hat{D}_{i}(G)=\frac{1}{2}(\sum_{i\in V_{1}}\hat{D}_{i}(G)+\sum_{j\in V_{2}}\hat{D}_{j}(G))\\
&\leq&\frac{1}{2}[\sum_{i\in V_{1}}(d_{i}+\frac{1}{3}(n-a-d_{i})+\frac{1}{2}(a-1))+\sum_{j\in V_{2}}(d_{j}+\frac{1}{3}(a-d_{j})+\frac{1}{2}(n-a-1))]\\
&=&\frac{2a^{2}-2na+3n^{2}-3n}{12}+\frac{1}{3}\sum_{i\in V(G)}d_{i}=\frac{2a^{2}-2na+3n^{2}-3n}{12}+\frac{2m}{3}\\
&\leq&\frac{2a^{2}-2na+3n^{2}-3n}{12}+\frac{2a(n-a)}{3}=\frac{-2a^{2}+2na+n^{2}-n}{4}\leq\frac{3n^{2}-2n}{8}.
\end{eqnarray*}
Note that $\frac{3n^{2}-2n}{8}<\frac{n^{2}-3n+7}{2}.$ We complete the proof.
\end{proof}

\begin{theorem}\label{th9}
Let $G$ be a connected graph on $n\geq8$ vertices with minimum degree $\delta\geq2$. If $H(\overline{G})\leq\frac{5n^{2}-23n+28}{2(n-1)},$
then $G$ is a pancyclic graph unless  $G\in\{K_{5}\vee6K_{1}, K_{3}\vee
(K_{2}+3K_{1}), K_{3}\vee (K_{1, 4}+K_{1}), K_{3}\vee (K_{1, 3}+K_{2}), (K_{2}\vee 2K_{1})\vee 5K_{1}, K_{4}\vee5K_{1}\}$.
\end{theorem}

\begin{proof}
Suppose that $G$ is not a pancyclic graph. Then
\begin{eqnarray}
\nonumber
H(\overline{G})&=&\frac{1}{2}\sum_{i=1}^{n}\hat{D}_{i}(\overline{G})\geq\frac{1}{2}\sum_{v\in V(G)}[d_{\overline{G}}(v)+\frac{1}{n-1}(n-1-d_{\overline{G}}(v))]
\\ \nonumber
&=&\frac{1}{2}\sum_{v\in V(G)}[1+\frac{n-2}{n-1}d_{\overline{G}}(v)]
\\ \nonumber
&=&\frac{n}{2}+\frac{n-2}{2(n-1)}\sum_{v\in V(G)}(n-1-d_{G}(v))
\\ \nonumber
&=&\frac{n(n-1)}{2}-\frac{n-2}{2(n-1)}\sum_{v\in V(G)}d_{G}(v)\\
&=&\frac{n(n-1)}{2}-\frac{n-2}{n-1}m.
\end{eqnarray}
Note that $H(\overline{G})\leq\frac{5n^{2}-23n+28}{2(n-1)}$, we have $m\geq\frac{n(n-1)^{2}}{2(n-2)}-\frac{5n^{2}-23n+28}{2(n-2)}=\binom{n-2}{2}+4.$
By Lemma \ref{le1}, we obtain that $G\in\mathbb{N}\mathbb{P}$ or $G$ is a bipartite graph.
Note that $n\geq8$, $H(\overline{K_{2}\vee (K_{n-4}+2K_{1})})\geq\frac{n^{2}-n-8}{4}>\frac{5n^{2}-23n+28}{2(n-1)}.$
However, by a direct calculation, for all $G\in\mathbb{N}\mathbb{P}\setminus K_{2}\vee (K_{n-4}+2K_{1}),$ $H(\overline{G})\leq\frac{5n^{2}-23n+28}{2(n-1)}.$
If $G$ is a bipartite graph, by Lemma \ref{le4}, we have $H(\overline{G})>\frac{5n^{2}-23n+28}{2(n-1)}$ for $n\geq8$. So we obtain this theorem.
\end{proof}

\begin{theorem}\label{th10}
Let $G$ be a connected graph of order $n\geq5$ with minimum degree $\delta\geq2$. If $\rho(G)\leq n+3-\frac{14}{n},$
then $G$ is a pancyclic graph unless $G\in \{K_{2}\vee (K_{1, 3}+K_{1}), K_{3}\vee 4K_{1}\}$.
\end{theorem}

\begin{proof}
Assume that $G$ is not a pancyclic graph. By Lemma \ref{le2} and (2) in the proof of Theorem \ref{th6}, we have
$$\rho(G)\geq\frac{2W(G)}{n}\geq\frac{2}{n}[n(n-1)-m]=2(n-1)-\frac{2}{n}m.$$
Since $\rho(G)\leq n+3-\frac{14}{n}$, we have $m\geq\binom{n-2}{2}+4$.
According to Lemma \ref{le1}, $G\in\mathbb{N}\mathbb{P}$ or $G$ is a bipartite graph.

For $G=K_{2}\vee(K_{n-4}+2K_{1})$, let $X=(x_{1}, x_{2}, \ldots , x_{n})^{T}$ be the eigenvector corresponding to $\rho(G)$, where $x_{i}~(1\leq i\leq n-4)$ corresponds to the vertex of degree $n-3$, $x_{j}~(n-3\leq j\leq n-2)$ corresponds to the vertex of degree $2$ and $x_{l}~(n-1\leq l\leq n)$ corresponds to the vertex of degree $n-1$. Without loss of generality, let $x:=x_{i}, 1\leq i\leq n-4$; $y:=x_{j}, n-1\leq j\leq n$; $z:=x_{l}, n-1\leq l\leq n$.
As $\rho X=D(G)X$, we have
\begin{equation*} \label{nu0}
\left\{
\begin{array}{ll}
\rho x=(n-5)x+4y+2z,\\
\rho y=2(n-4)x+2y+2z,\\
\rho z=(n-4)x+2y+z.
\end{array}
\right.
\end{equation*}
It follows that  $\rho(K_{2}\vee(K_{n-4}+2K_{1}))$ is the largest root of the equation
$$\rho^{3}-(n-2)\rho^{2}-(7n-23)\rho-2(n-3)=0.$$
Let $f(x)=\rho^{3}-(n-2)\rho^{2}-(7n-23)\rho-2(n-3),$ then $n\geq5,$ $f(n+3-\frac{14}{n})=-2n^{2}+16n-6-\frac{476}{n}+\frac{2156}{n^{2}}-\frac{2744}{n^{3}}<0.$
Hence $\rho(K_{2}\vee(K_{n-4}+2K_{1}))>n+3-\frac{14}{n}$ for $n\geq5.$ By a direct calculation, we have the following Table 1.

\begin{center}
\begin{tabular}{c|c|c}
$G$ & $\rho(G)$ & $n+3-\frac{14}{n}$\\
\hline
$K_{5}\vee
6K_{1}$  & 13.245 & 12.727\\
\hline
$K_{3}\vee
(K_{2}+3K_{1})$  & 9.5947 & 9.25\\
\hline
$K_{3}\vee (K_{1,4}+K_{1})$  & 10.8341 & 10.444\\
\hline
$K_{3}\vee(K_{1,3}+K_{2})$ & 10.7672 & 10.444\\
\hline
$(K_{2}\vee 2K_{1})\vee 5K_{1}$ & 10.7624 & 10.444\\
\hline
$K_{4}\vee 5K_{1}$  & 10.6325 & 10.444\\
\hline
$K_{1,2}\vee 4K_{1}$ & 8.2736 & 8\\
\hline
$K_{2}\vee(K_{1,3}+K_{1})$ & 8 & 8\\
\hline
$K_{3}\vee4K_{1}$ &  8 & 8\\
\multicolumn{3}{c}{Table 1. The distance spectral radius of some graphs}
\end{tabular}
\end{center}
If $G$ is a bipartite graph, by (1) the proof of Lemma \ref{le4}, we have $W(G)\geq \frac{3n^{2}-4n}{4},$ then $\rho(G)\geq \frac{2W(G)}{n}\geq \frac{3n-4}{2}>n+3-\frac{14}{n}.$
This completes the proof of Theorem \ref{th10}.
\end{proof}

\begin{theorem}\label{th11}
Let $G$ be a connected graph on $n\geq8$ vertices with $\delta\geq2$.
If ${\rho}^{\star}(\overline{G})\leq\frac{5n^{2}-23n+28}{n(n-1)},$ then $G$ is a pancyclic graph.
\end{theorem}
\begin{proof}
Suppose that $G$ is not a pancyclic graph. From Lemma \ref{le3} and (3) in the proof of Theorem \ref{th9}, we have
$${\rho}^{\star}(\overline{G})\geq\frac{2H(\overline{G})}{n}\geq n-1-\frac{2(n-2)m}{n(n-1)}.$$
Since ${\rho}^{\star}(\overline{G})\leq\frac{5n^{2}-23n+28}{n(n-1)}$, we have $m\geq\binom{n-2}{2}+4$.
According to Lemma \ref{le1}, we obtain $G\in\mathbb{N}\mathbb{P}$ or $G$ is a bipartite graph.

For $G=K_{2}\vee(K_{n-4}+2K_{1})$, let $X=(x_{1}, x_{2}, \cdots , x_{n})^{T}$ be the eigenvector corresponding to $\rho^{\star}(\overline{G})$, where $x_{i}~ (1\leq i\leq n-4)$ corresponds to the vertex of degree $2$, $x_{j}~(n-3\leq j\leq n-2)$ corresponds to the vertex of degree $n-3$ and $x_{l}~(n-1\leq l\leq n)$ corresponds to the vertex of degree $0$. Without loss of generality, let $x:=x_{i}, 1\leq i\leq n-4$; $y:=x_{j}, n-1\leq j\leq n$; $z:=x_{l}, n-1\leq l\leq n$.
As $\rho X=RD(G)X$, we have
\begin{equation*} \label{nu0}
\left\{
\begin{array}{ll}
\rho^{\star} x=\frac{1}{2}(n-5)x+2y+0z,\\
\rho^{\star} y=(n-4)x+y+0z.
\end{array}
\right.
\end{equation*}
It follows that  $\rho(K_{2}\vee(K_{n-4}+2K_{1}))$ is the largest root of the equation
$$2(\rho^{\star})^{2}-(n-3)\rho^{\star}+11-3n=0.$$
Then $\rho^{\star}=\frac{n-3+\sqrt{n^{2}+18n-79}}{4}>\frac{5n^{2}-23n+28}{n(n-1)}$ for $n\geq5.$ By a direct calculation, we have the following Table 2.
\begin{center}
\begin{tabular}{c|c|c}
$G$ 
& ${\rho}^{\star}(\overline{G})$ & $\frac{5n^{2}-23n+28}{n(n-1)}$\\
\hline
$K_{5}\vee
6K_{1}$  & 5 & 3.4546\\
\hline
$K_{3}\vee
(K_{2}+3K_{1})$  & 4.0663 & 2.9286\\
\hline
$K_{3}\vee (K_{1,4}+K_{1})$  & 4.4115 & 3.1389\\
\hline
$K_{3}\vee(K_{1,3}+K_{2})$ & 4.5455 & 3.1389\\
\hline
$(K_{2}\vee 2K_{1})\vee 5K_{1}$ & 4 & 3.1389\\
\hline
$K_{4}\vee 5K_{1}$  & 4 & 3.1389\\
\hline
$K_{1,2}\vee 4K_{1}$ & 3 & 2.667\\
\hline
$K_{2}\vee(K_{1,3}+K_{1})$ & 3.4575 & 2.667\\
\hline
$K_{3}\vee4K_{1}$ &  3 & 2.667\\
\multicolumn{3}{c}{Table 2. The Harary spectral radius of complements of some
graphs}
\end{tabular}
\end{center}
Clearly, for $G\in\mathbb{N}\mathbb{P},$ ${\rho}^{\star}(\overline{G})>\frac{5n^{2}-23n+28}{n(n-1)}.$

If $G$ is a bipartite graph of order $n\geq8$ whose vertices are divided into two disjoint and independent sets $V_{1}$ and $V_{2}$, $|V_{1}|=a, (1\leq a\leq n-1),$ $|V_{2}|=n-a.$

\noindent {\bf Case 1.} $G$ is not a complete bipartite graph.

By Lemma \ref{le3} and the proof of Lemma \ref{le5}, we have $\rho^{\star}(\overline{G})\geq\frac{3n-4}{4}>\frac{5n^{2}-23n+28}{n(n-1)}$ for $n\geq3.$

\noindent {\bf Case 2.} $G$ is a complete bipartite graph.

Then $G=K_{a, n-a}$.

By Lemma \ref{le3} and the proof of Lemma \ref{le5}, we have $\rho^{\star}(\overline{G})\geq\frac{n-2}{2}>\frac{5n^{2}-23n+28}{n(n-1)}$ for $n\geq8.$
Hence we obtain the theorem.
\end{proof}

\section{\mathversion{bold} Conclusions}
\noindent We consider the problem of deciding whether a given graph is pancyclic. We present some sufficient conditions for a graph to be pancyclic in terms of the Wiener index, the Harary index, the distance spectral radius and the Harary spectral radius of a graph, respectively.

\small {

}


\begin{thebibliography}{99}
\bibitem{FLH} Feng, L, Zhu, X, Liu, W: Wiener index, Harary index and graph properties, {\it Discrete Appl. Math.} {\bf223}, 72-83 (2017)

\bibitem{HH} Hua, H, Wang, M: On Harary index and traceable graphs, {\it MATCH Commun. Math. Comput. Chem.} {\bf70}, 297-300 (2013)

\bibitem{HN} Hua, H, Ning, B: Wiener Index, Harary Index and Hamiltonicity of Graphs, {\it MATCH Commun. Math. Comput. Chem.} {\bf78}, 153-162 (2017)

\bibitem{IG} Indulal, G: Sharp bounds on the distance spectral radius and the distance energy of graphs, {\it Linear Algebra Appl.} {\bf430}, 106-113 (2009)

\bibitem{IO} Ivanciuc, O, Balaban, T, Balaban, A: Reciprocal distance matrix, related local vertex invariants
and topological indices, {\it J. Math. Chem.} {\bf12}, 309-318 (1993)

\bibitem{DX3} Jia, H, Liu, R, Du, X: Wiener index and Harary index on Hamilton-connected and traceable graphs, {\it Ars Combinatoria} {\bf141}, 53-62 (2018)

\bibitem{DH} Kuang, M, Huang, G, Deng, H: Some sufficient conditions for Hamiltonian property in terms of Wiener-type invariants, {\it Proceedings Mathematical Sciences} {\bf126}, 1-9 (2016)

\bibitem{LR1} Li, R: Wiener index and some Hamiltonian properties of graphs, {\it International Journal of
Mathematics and Soft Computing} {\bf5}, 11-16 (2015)

\bibitem{LR2} Li, R: Harary index and some Hamilitonian properties of graphs, {\it AKCE International Journal of Graphs and
Combinatorics} {\bf12}, 64-69 (2015)

\bibitem{DX1} Liu, R, Du, X, Jia, H: Some observations on Harary index and traceable graphs, {\it MATCH Commun. Math. Comput. Chem.} {\bf77}, 195-208 (2017)

\bibitem{DX2} Liu, R, Du, X, Jia, H: Wiener index on traceable and Hamiltonian graphs, {\it Bull. Aust. Math. Soc.} {\bf94}, 362-372 (2016)

\bibitem{PD} Plav\v{s}i\'{c}, D,  Nikoli\'{c}, S, Trinajsti\'{c}, N, Mihali\'{c}, Z: On the Harary index for the characterization
of chemical graphs, {\it J. Math. Chem.} {\bf12}, 235-250 (1993)

\bibitem{WH} Wiener, H: Structural determination of paraffin boiling points, {\it Journal of the American
Chemical Society} {\bf69}, 17-20 (1947)

\bibitem{YYSX} Yu, G, Yu, T, Shu, A, Xia, X: Some sufficient conditions on pancyclic graphs, {\it arXiv: 1809.09897v1 [math.CO] 26 Sep 2018.}

\bibitem{YL} Yang, L: Wiener index and traceable graphs, {\it Bull. Aust. Math. Soc.} {\bf88}, 380-383 (2013)

\bibitem{ZT} Zeng, T: Harary index and Hamiltonian property of graphs, {\it MATCH Commun. Math. Comput. Chem} {\bf70}, 645-649 (2013)


\end{thebibliography}
\end{document}